\documentclass[a4]{amsart}

\usepackage{amssymb}
\usepackage[all]{xy}
\usepackage{amsmath, amsfonts, amsthm , amssymb} 

\input xypic

\newtheorem{theorem}{Theorem}[section]
\newtheorem{proposition}[theorem]{Proposition}
\newtheorem{lemma}[theorem]{Lemma}
\newtheorem{corollary}[theorem]{Corollary}

\theoremstyle{definition}

\newcommand{\uxa}{\ensuremath{(\underline{X},\underline{A})}}
\newcommand{\cxx}{\ensuremath{(\underline{CX},\underline{X})}} 
\newcommand{\cyy}{\ensuremath{(\underline{CY},\underline{Y})}}

\newcommand{\starr}{\ensuremath{\mbox{star}}}
\newcommand{\link}{\ensuremath{\mbox{link}}} 
\newcommand{\starrK}{\ensuremath{\mbox{star}_{\,\overline{K}\,}}}
\newcommand{\linkK}{\ensuremath{\mbox{link}_{\,\overline{K}\,}}}

\newcommand{\zk}{\ensuremath{\mathcal{Z}_{K}}}

\newcommand{\hlgy}[1]{\ensuremath{H_{*}(#1)}}

\newcounter{bean}
\newenvironment{letterlist}{\begin{list}{\rm ({\alph{bean}})}
      {\usecounter{bean}\setlength{\rightmargin}{\leftmargin}}}
      {\end{list}}

\newcommand{\namedright}[3]{\ensuremath{#1\stackrel{#2}
 {\longrightarrow}#3}}
\newcommand{\nameddright}[5]{\ensuremath{#1\stackrel{#2}
 {\longrightarrow}#3\stackrel{#4}{\longrightarrow}#5}}
\newcommand{\namedddright}[7]{\ensuremath{#1\stackrel{#2}
 {\longrightarrow}#3\stackrel{#4}{\longrightarrow}#5
  \stackrel{#6}{\longrightarrow}#7}}

\newcommand{\qqed}{\hfill\Box}

\begin{document}
\title{Moore's conjecture for polyhedral products} 
\author{Yanlong Hao$^{\ast}$} 
\address{School of Mathematical Sciences and LPMC, Nankai University, 
    Tianjin 300071, P.R. China} 
\email{haoyanlong13@mail.nankai.edu.cn} 
\author{Qianwen Sun$^{\ast}$} 
\address{School of Mathematical Sciences and LPMC, Nankai University, 
   Tianjin 300071, P.R. China} 
\email{qwsun13@mail.nankai.edu.cn} 
\author{Stephen Theriault}
\address{Mathematical Sciences, University of Southampton, 
   Southampton SO17 1BJ, United Kingdom}
\email{S.D.Theriault@soton.ac.uk} 

\thanks{$^{*}$ Research supported by the National Natural Science Foundation 
   of China (No. 11571186)}

\date{}

\begin{abstract} 
Moore's Conjecture is shown to hold for generalized moment-angle 
complexes and a criterion is proved that determines when a polyhedral 
product is elliptic or hyperbolic. 
\end{abstract}

\maketitle 

\section{Introduction} 
\label{sec:intro} 

Moore's Conjecture envisions a deep relationship between the rational 
and torsion homotopy groups of finite $CW$-complexes. Let $X$ be a 
finite $CW$-complex. The \emph{homotopy exponent} of $X$ at a prime $p$ 
is the least power of $p$ that annihilates the $p$-torsion in the homotopy 
groups of $X$. The space $X$ is \emph{elliptic} if it has finitely many 
rational homotopy groups and \emph{hyperbolic} if it has infinitely many 
rational homotopy groups. 
\medskip 

\noindent 
\textbf{Moore's Conjecture}:  
   Let $X$ be a finite, simply-connected $CW$-complex. Then the following are equivalent: 
   \begin{letterlist} 
      \item $X$ is elliptic; 
      \item $X$ has a finite homotopy exponent at every prime $p$; 
      \item $X$ has a finite homotopy exponent at some prime $p$. 
   \end{letterlist} 
\medskip 
The conjecture posits that the nature of the rational homotopy groups should 
have a profound impact on the nature of the, seemingly unrelated, 
torsion homotopy groups, and that torsion behaviour at one prime has a profound  
impact on torsion behaviour at all primes. The conjecture has been shown to 
hold in a number of cases. Elliptic spaces with finite exponents at all primes 
include spheres~\mbox{\cite{J,To}}, finite $H$-spaces~\cite{Lo}, $H$-spaces with finitely 
generated cohomology~\cite{CPSS}, and odd primary Moore spaces~\cite{N}. 
Hyperbolic spaces with no exponent at any prime include wedges of 
simply-connected spheres, most torsion-free two-cell complexes~\cite{NS}, and 
torsion-free suspensions~\cite{Se}. There are also partial results. In~\cite{MW} 
it was shown that if $X$ is elliptic then it has an exponent at all but 
finitely many primes, in~\cite{St} it was shown that if $X$ is hyperbolic 
and $\hlgy{\Omega X;\mathbb{Z}}$ is $p$-torsion free then, provided $p$ 
is large enough, $X$ has no exponent at $p$, and in~\cite{A} Moore's Conjecture 
was shown to hold for all but finitely many primes in the case of spaces 
having Lusternik-Schnirelmann category two.

Moore's conjecture is also related to an important phenomenon in rational 
homotopy theory. F\'{e}lix, Halperin and Thomas~\cite{FHT} proved the remarkable 
fact that a finite $CW$-complex is either elliptic or its total number of rational 
homotopy groups below dimension $n$ grows exponentially with~$n$. There 
is no hyperbolic space whose rational homotopy groups have polynomial growth. 

In this paper we consider Moore's conjecture, and the notions of being 
elliptic or hyperbolic, in the context of polyhedral products. 
Let $K$ be a simplicial complex on $m$ vertices.  For $1\leq i\leq m$,
let $(X_{i},A_{i})$ be a pair of pointed $CW$-complexes, where $A_{i}$ is 
a pointed subspace of $X_{i}$. Let $\uxa=\{(X_{i},A_{i})\}_{i=1}^{m}$ be 
the sequence of $CW$-pairs. For each simplex (face) $\sigma\in K$, let 
$\uxa^{\sigma}$ be the subspace of $\prod_{i=1}^{m} X_{i}$ defined by
\[\uxa^{\sigma}=\prod_{i=1}^{m}\overline{X}_{i}\qquad
       \mbox{where}\qquad\overline{X}_{i}=\left\{\begin{array}{ll}
                                             X_{i} & \mbox{if $i\in\sigma$} \\
                                             A_{i} & \mbox{if $i\notin\sigma$}.
                                       \end{array}\right.\]
The \emph{polyhedral product} determined by \uxa\ and $K$ is
\[\uxa^{K}=\bigcup_{\sigma\in K}\uxa^{\sigma}\subseteq\prod_{i=1}^{m} X_{i}.\]
The topology of polyhedral products has received a great deal of attention 
recently due to their central role in toric topology~\cite{BBCG1,BP,GT1,GT2,IK}. 
Important special cases include \emph{moment-angle complexes} $\zk$,  
when each pair $(X_{i},A_{i})$ equals $(D^{2},S^{1})$, and \emph{generalized 
moment-angle complexes} $\zk(D^{n},S^{n-1})$, when each pair $(X_{i},A_{i})$ 
equals $(D^{n},S^{n-1})$ for $n\geq 2$. By~\cite{BP}, each generalized 
moment-angle complex is a finite, simply-connected $CW$-complex, so Moore's Conjecture 
may be considered. 

To state our results some definitions are needed. Write $[m]$ for the vertex 
set $\{1,\ldots,m\}$. Let~$\Delta^{m-1}$ be the standard $m$-simplex with 
vertex set $[m]$. The faces of~$\Delta^{m-1}$ can be identified with sequences 
$(i_{1},\ldots,i_{k})$ for $1\leq i_{1}<\cdots <i_{k}\leq m$. If $K$ is a simplicial 
complex on the vertex set $[m]$ then a sequence $\sigma=(i_{1},\ldots,i_{k})$ 
is a \emph{missing face} of $K$ if $\sigma\notin K$. It is a \emph{minimal 
missing face} of $K$ if no proper subsequence of $\sigma$ is a missing face of $K$.  

\begin{theorem} 
   \label{intromacmoore} 
   Let $K$ be a simplicial complex on the vertex set $[m]$ and let $\uxa$ 
   be any sequence of pairs $(D^{n_{i}},S^{n_{i}-1})$ with $n_{i}\geq 2$ 
   for $1\leq i\leq m$. Then: 
   \begin{letterlist} 
      \item $\uxa^{K}$ is elliptic if and only if the minimal missing 
                faces of $K$ are mutually disjoint; 
      \item Moore's conjecture holds for $\uxa^{K}$. 
   \end{letterlist} 
\end{theorem} 

In particular, Theorem~\ref{intromacmoore} includes generalized moment-angle 
complexes $\zk(D^{n},S^{n-1})$ for $n\geq 2$ as a special case. Part~(a) of 
Theorem~\ref{intromacmoore} was proved by~\cite{BBCG1} in the special case 
of the moment-angle complex~$\zk$, although part~(b) was not. The restriction 
to $n\geq 2$ is made to ensure that certain retractions constructed in 
Theorem~\ref{moore} involve wedges of simply-connected spheres which 
are hyperbolic, rather than wedge of circles which are Eilenberg-MacLane spaces.  

We also give a general criterion for when a polyhedral product is elliptic or 
hyperbolic. This generalizes and reformulates in more combinatorial terms 
results obtained by F\'{e}lix and Tanr\'{e}~\cite{FT}.  

\begin{theorem} 
   \label{introellhyp} 
   Let $K$ be a simplicial complex on the vertex set $[m]$ and let $\uxa$ 
   be any sequence of pointed, path-connected pairs. For $1\leq i\leq m$, let $Y_{i}$ 
   be the homotopy fibre of the inclusion 
   \(\namedright{A_{i}}{}{X_{i}}\) 
   and suppose that each $Y_{i}$ is rationally nontrivial. Then the polyhedral 
   product $\uxa^{K}$ is elliptic if and only if three conditions hold: 
   \begin{itemize} 
      \item[(i)] each $X_{i}$ is elliptic; 
      \item[(ii)] all the minimal missing faces of $K$ are mutually disjoint; 
      \item[(iii)] if $v$ is a vertex of a minimal missing face of $K$ then 
                      $Y_{v}$ is rationally homotopy equivalent to a sphere.  
   \end{itemize} 
\end{theorem}  

For example, let $K$ be the boundary of a pentagon. A result essentially due 
to MacGavran~\cite{Mac} shows that $\zk$ is diffeomorphic to the connected 
sum of $5$ copies of $S^{3}\times S^{4}$. It is well known that such a connected 
sum is hyperbolic. If the ingredient pairs of spaces change from 
$(D^{2},S^{1})$ to $(C\mathbb{C}P^{n},\mathbb{C}P^{n})$, where $C\mathbb{C}P^{n}$ 
is the cone on $\mathbb{C}P^{n}$, then the polyhedral product 
$(\underline{C\mathbb{C}P^{n}},\underline{\mathbb{C}P^{n}})^{K}$ is 
some analogue of a connected sum, but its homotopy type is not clear. 
Observe that each $Y_{i}$ in this case is $\Omega\mathbb{C}P^{n}$, so is 
rationally nontrivial, and the minimal missing faces of $K$ are not mutually disjoint. 
So Theorem~\ref{introellhyp} implies that 
$(\underline{C\mathbb{C}P^{n}},\underline{\mathbb{C}P^{n}})^{K}$ is hyperbolic. 
The most pertinent point here is that this determination is made without any 
reference to the standard differential graded Lie algebra tools commonly used 
to decide ellipticity or hyperbolicity. 

The authors would like to thank the referee for suggestions that 
improved the clarity of the paper.

\section{Polyhedral product ingredients} 
\label{sec:polyprods} 

This section contains the properties of polyhedral products that will 
be needed. The main results are Theorems~\ref{cyyfib} and~\ref{cyysplit}, which are 
of independent interest. 

\begin{theorem} 
   \label{cyyfib} 
   Let $K$ be a simplicial complex on the vertex set $[m]$ and let 
   $\uxa$ be any sequence of pointed, path-connected $CW$-pairs. Then 
   there is a homotopy fibration 
   \[\nameddright{\cyy^{K}}{}{\uxa^{K}}{}{\prod_{i=1}^{m} X_{i}}\]  
   where, for $1\leq i\leq m$, $Y_{i}$ is the homotopy fibre of the inclusion 
   \(\namedright{A_{i}}{}{X_{i}}\). 
\end{theorem} 

\begin{proof} 
In general, let 
\(f\colon\namedright{B}{}{Z}\) 
be a pointed, continuous map, let $I=[0,1]$ be the unit interval with basepoint $0$, 
and let $PZ$ be the path space of $Z$. Then the homotopy fibre of $f$ is the 
pullback of $f$ and the evaluation map 
\(ev_{1}\colon\namedright{PZ}{}{Z}\), 
where $ev_{1}(\omega)=\omega(1)$. In our case, we are given $m$ pairs 
of spaces $(X_{i},A_{i})$. The homotopy pullback of the identity map 
\(1_{X_{i}}\colon\namedright{X_{i}}{}{X_{i}}\) 
is $PX_{i}$, and~$Y_{i}$ is defined as the homotopy pullback of the inclusion 
\(j_{i}\colon\namedright{A_{i}}{}{X_{i}}\). 
Note that as $j_{i}$ is a subspace inclusion then $Y_{i}$ is the inverse 
image $ev_{1}^{-1}(A_{i})\subseteq PX_{i}$.   

Let $\sigma$ be a face in $K$. By the definition of a 
polyhedral product, $\uxa^{\sigma}=\overline{X}_{1}\times\cdots\times\overline{X}_{m}$, 
where~$\overline{X}_{i}$ is $X_{i}$ if $i\in\sigma$ and is $A_{i}$ if $i\notin\sigma$. 
Consider the homotopy pullback 
\[\diagram 
      Q\rto\dto & P(X_{1}\times\cdots\times X_{m})\dto^{ev_{1}} \\ 
      \uxa^{\sigma}\rto^-{i^{\sigma}} & X_{1}\times\cdots\times X_{m}  
  \enddiagram\] 
where $i^{\sigma}$ is the inclusion. 
Observe that $P(X_{1}\times\cdots\times\cdots X_{m})$ is homeomorphic to 
$PX_{1}\times\cdots\times PX_{m}$ and under this homeomorphism $ev_{1}$ 
translates into a product of the $m$ evaluation maps $ev_{1}$ on each $PX_{i}$.  
As the product of pullbacks is a pullback, we see that $Q$ is homeomorphic to 
$\overline{Y}_{1}\times\cdots\times\overline{Y}_{m}$, where~$\overline{Y}_{i}$ 
is $PX_{i}$ if $i\in\sigma$ and is $Y_{i}$ if $i\notin\sigma$. That is, 
$Q$ is homeomorphic to $(\underline{PX},\underline{Y})^{\sigma}$. Moreover, 
since $i^{\sigma}$ is an inclusion, $(\underline{PX},\underline{Y})^{\sigma}$ is 
the inverse image 
$(ev_{1}\times\cdots\times ev_{1})^{-1}(\uxa^{\sigma})\subseteq PX_{1}\times\cdots\times PX_{m}$. 
Now $\uxa^{K}$ is the union 
of the spaces $\uxa^{\sigma}$ for all $\sigma\in K$, where intersections have 
been identified. Since inverse images preserve unions and intersections, we 
obtain that the homotopy fibre of the inclusion 
\(\namedright{\uxa^{K}}{}{X_{1}\times\cdots\times X_{m}}\) 
is homeomorphic to $(\underline{PX},\underline{Y})^{K}$. 

Finally, since $PX_{i}$ is contractible, the inclusion 
\(\namedright{Y_{i}}{}{PX_{i}}\) 
extends to a map 
\(\namedright{CY_{i}}{}{PX_{i}}\) 
which is a homotopy equivalence. Thus the induced map of pairs 
\(\namedright{(CY_{i},Y_{i})}{}{(PX_{i},Y_{i})}\) 
is a homotopy equivalence. Hence the homotopy fibre of the inclusion 
\(\namedright{\uxa^{K}}{}{X_{1}\times\cdots\times X_{m}}\) 
is homotopy equivalent to $(\underline{CY},\underline{Y})^{K}$. 
\end{proof} 

Next, we show that the homotopy fibration in Theorem~\ref{cyyfib} splits after looping.  
Let $K$ be a simplicial complex on the vertex set $[m]$. If $I\subseteq [m]$ 
then the \emph{full subcomplex} $K_{I}$ of~$K$ is defined as the simplicial 
complex 
\[K_{I}=\bigcup\{\sigma\in K\mid\mbox{the vertex set of $\sigma$ is in $I$}\}.\] 
The definition of $K_{I}$ implies that the inclusion 
\(\namedright{K_{I}}{}{K}\) 
is a map of simplicial complexes. This induces a map of polyhedral products 
\(\namedright{\uxa^{K_{I}}}{}{\uxa^{K}}\). 
There is no retraction of $K_{I}$ off $K$ as simplicial complexes, however, 
in~\cite{DS} it was shown that there is nevertheless a retraction 
of $\uxa^{K_{I}}$ off $\uxa^{K}$. 

\begin{proposition} 
   \label{fullsubcomplex} 
   Let $K$ be a simplicial complex on the vertex set $[m]$ and let $\uxa$ 
   be any sequence of pointed, path-connected $CW$-pairs. Let $I\subseteq [m]$. 
   Then the inclusion 
   \(\namedright{\uxa^{K_{I}}}{}{\uxa^{K}}\) 
   has a left inverse.~$\qqed$ 
\end{proposition} 

\begin{theorem} 
   \label{cyysplit} 
   Let 
   \(\nameddright{\cyy^{K}}{}{\uxa^{K}}{}{\prod_{i=1}^{m} X_{i}}\) 
   be the homotopy fibration in Theorem~\ref{cyyfib}. Then there is a homotopy equivalence 
   \[\Omega\uxa^{K}\simeq(\prod_{i=1}^{m}\Omega X_{i})\times\Omega\cyy^{K}.\] 
\end{theorem} 

\begin{proof} 
For $1\leq i\leq m$, let $I_{i}=\{i\}$. Observe that the 
full subcomplex $K_{I_{i}}$ of $K$ is just the vertex $\{i\}$. By the 
definition of the polyhedral product, $\uxa^{K_{I_{i}}}=X_{i}$. 
Proposition~\ref{fullsubcomplex} therefore implies that $X_{i}$ retracts 
off $\uxa^{K}$. Explicitly, the composite 
\(\namedddright{X_{i}=\uxa^{K_{I_{i}}}}{}{\uxa^{K}}{}{\prod_{i=1}^{m} X_{i}}{\mbox{proj}}{X_{i}}\) 
is the identity map. After looping, the loop maps 
\(\namedright{\Omega X_{i}}{}{\Omega\uxa^{K}}\) 
may be multiplied together to obtain a map 
\(\namedright{\prod_{i=1}^{m}\Omega X_{i}}{}{\Omega\uxa^{K}}\) 
which is a right homotopy inverse of the map 
\(\namedright{\Omega\uxa^{K}}{}{\prod_{i=1}^{m}\Omega X_{i}}\). 
Hence, if $\mu$ is the loop multiplication on $\Omega\uxa^{K}$, then 
the composite 
\[\nameddright{(\prod_{i=1}^{m}\Omega X_{i})\times\Omega\cyy^{K}}{}  
      {\Omega\uxa^{K}\times\Omega\uxa^{K}}{\mu}{\Omega\uxa^{K}}\] 
is a homotopy equivalence. 
\end{proof} 

Theorem~\ref{cyysplit} implies that homotopy group information 
about $\uxa^{K}$ is determined by that of the ingredient spaces $X_{i}$ and 
$\cyy^{K}$. This is useful because the spaces $\cyy^{K}$ are much better 
understood than the spaces $\uxa^{K}$. 

This section concludes with the statement of two other results which will 
be used later. The first, proved in~\cite{GT2}, relates pushouts 
of simplicial complexes to pushouts of polyhedral products. 

\begin{proposition}
   \label{uxaunion}
   Let $K$ be a simplicial complex on the vertex set $[m]$. Suppose that 
   there is a pushout of simplicial complexes 
   \[\diagram
           L\rto\dto & K_{2}\dto \\
           K_{1}\rto & K
     \enddiagram\] 
   Let $L^{\circ}$, $K^{\circ}_{1}$ and $K^{\circ}_{1}$ be $L$, 
   $K_{1}$ and $K_{2}$ regarded as simplicial complexes on the same 
   vertex set as $K$. Then there is a pushout of polyhedral products 
   \[\diagram
           \uxa^{L^{\circ}}\rto\dto & \uxa^{K^{\circ}_{2}}\dto \\
           \uxa^{K^{\circ}_{1}}\rto & \uxa^{K}.
     \enddiagram\] 
\end{proposition} 
\vspace{-1cm}~$\qqed$\medskip 

Second, we give two examples where the homotopy type of $\cyy^{K}$ 
is explicitly identified. Part~(a) in Lemma~\ref{2egs} is immediate from 
the definition of the polyhedral product, while part~(b) was proved by Porter~\cite{P} 
when each $Y_{i}$ is a loop space and more generally in~\cite{GT2}.  

\begin{lemma} 
   \label{2egs} 
   Let $Y_{1},\ldots,Y_{m}$ be path-connected spaces. Then the following hold: 
   \begin{letterlist} 
      \item $\cyy^{\Delta^{m-1}}=\prod_{i=1}^{m} CY_{i}$; 
      \item $\cyy^{\partial\Delta^{m-1}}\simeq\Sigma^{m-1} Y_{1}\wedge\cdots\wedge Y_{m}$. 
   \end{letterlist} 
\end{lemma} 
\vspace{-1mm}~$\qqed$\medskip

\section{Combinatorial ingredients} 
\label{sec:combinatorics} 

This section records the combinatorial information that will be needed. 
Let $K$ be a simplicial complex on the index set $[m]$. For a vertex $v\in K$, 
the \emph{star}, \emph{restriction} (or \emph{deletion}) and \emph{link} 
of $v$ are the subcomplexes
\[\begin{array}{rcl}
     \starr_{K}(v) & = & \{\tau\in K\mid \{v\}\cup\tau\in K\}; \\
     K\backslash v & =
          & \{\tau\in K\mid \{v\}\cap\tau=\emptyset\}; \\
     \link_{K}(v) & = & \starr_{K}(v)\cap K\backslash v.
   \end{array}\] 
The \emph{join} of two simplicial complexes $K_{1},K_{2}$ on disjoint 
index sets is the simplicial complex 
\[K_{1}\ast K_{2}=\{\sigma_{1}\cup \sigma_{2}\mid \sigma_{i}\in K_{i}\}.\] 
From the definitions, it follows that $\starr_{K}(v)$ is a join,  
\[\starr_{K}(v)=\{v\}\ast\link_{K}(v),\] 
and there is a pushout
\[\diagram
         \link_{K}(v)\rto\dto & \starr_{K}(v)\dto \\
         K\backslash v\rto & K.
  \enddiagram\] 

A \emph{face} of $K$ is a simplex of $K$. Let $\Delta^{m-1}$ be the standard 
$m$-simplex on the vertex set $[m]$ and note that $K$ is a subcomplex 
of $\Delta^{m-1}$. Recall from the Introduction that a face $\sigma\in\Delta^{m-1}$ 
is a \emph{missing face} of $K$ if $\sigma\notin K$. It is a 
\emph{minimal missing face} if any proper face of $\sigma$ is a face of $K$. 
Denote the set of minimal missing faces of $K$ by $MMF(K)$. For a simplex $\sigma$, 
let $\partial\sigma$ be its boundary. Observe that $\sigma\in MMF(K)$ if and 
only if $\sigma\notin K$ but $\partial\sigma\subseteq K$. 

There is a special case which will play a crucial role in what follows. 
Let $\overline{K}$ be a simplicial complex on the vertex set $[m]$ with the 
property that it has precisely two distinct minimal missing faces and these 
have non-empty intersection. That is, suppose that 
$MMF(\overline{K})=\{\sigma_{1},\sigma_{2}\}$ where $\sigma_{1}$ 
and $\sigma_{2}$ have vertex sets $I$ and $J$ respectively, satisfying  
$I\neq J$, $I\cup J=[m]$ and $I\cap J\neq\emptyset$. Let $w$ be a vertex 
in both $I$ and $J$. 

Consider the star-link-restriction pushout of $\overline{K}$ with respect to 
the vertex $w$:  
\begin{equation} 
  \label{slr} 
  \diagram
         \linkK(w)\rto\dto & \starrK(w)\dto \\
         \overline{K}\backslash w\rto & \overline{K}.
  \enddiagram 
\end{equation} 
Let $\overline{\sigma}_{1}$ and $\overline{\sigma}_{2}$ be the proper faces 
of $\sigma_{1}$ and $\sigma_{2}$ on the vertex sets 
$\overline{I}=I\backslash\{w\}$ and $\overline{J}=J\backslash\{w\}$ respectively. 
 
\begin{lemma} 
   \label{mfstar} 
   We have 
   $\overline{\sigma}_{1},\overline{\sigma}_{2}\in MMF(\starrK(w))$. 
\end{lemma} 

\begin{proof} 
Consider $\overline{\sigma}_{1}$, the argument for $\overline{\sigma}_{2}$ being 
similar. First we show that $\overline{\sigma}_{1}$ is a missing face of 
$\starrK(w)$. For if $\overline{\sigma}_{1}\in\starrK(w)$ 
then, as $w$ is not a vertex of $\overline{\sigma}_{1}$, we also have 
$\overline{\sigma}_{1}\in\overline{K}\backslash w$, implying that 
$\overline{\sigma}_{1}\in\linkK(w)= 
     \starrK(w)\cap\overline{K}\backslash w$. 
This in turn implies that 
$\overline{\sigma}_{1}\ast\{w\}\in\starrK(w)$. But 
$\overline{\sigma}_{1}\ast\{w\}=\sigma_{1}$, so 
$\sigma_{1}\in\starrK(w)$. Therefore, by~(\ref{slr}), 
$\sigma_{1}\in\overline{K}$, contradicting the fact that $\sigma_{1}$ is a 
missing face of $\overline{K}$. 

Next, we show that that $\overline{\sigma}_{1}$ 
is a minimal missing face of $\starrK(w)$. If not, then some 
proper face $\tau$ of $\overline{\sigma}_{1}$ is also a missing face 
of $\starrK(w)$. As $w$ is not a vertex of $\overline{\sigma}_{1}$, 
it is not a vertex of $\tau$ either. Therefore $\tau\ast\{w\}$ is a 
missing face of $\starrK(w)$. The presence of the vertex $w$ 
in $\tau\ast\{w\}$ implies that it is also not a face of~$\overline{K}\backslash w$. 
On the other hand, by~(\ref{slr}), $\overline{K}$ is the union of 
$\starrK(w)$ and $\overline{K}\backslash w$, so a 
face that is missing from both $\starrK(w)$ and 
$\overline{K}\backslash w$ must also be missing from $\overline{K}$. 
Therefore $\tau\ast\{w\}$ is a missing face of $\overline{K}$. But as 
$\tau$ is a proper face of $\overline{\sigma}_{1}$, $\tau\ast\{w\}$ is  
a proper face of $\overline{\sigma}_{1}\ast\{w\}=\sigma_{1}$, 
contradicting the fact that $\sigma_{1}$ is a minimial missing face of $\overline{K}$. 
\end{proof} 

\begin{corollary} 
   \label{mflink} 
   We have 
   $\partial\,\overline{\sigma}_{1},\partial\,\overline{\sigma}_{2}\subseteq 
         \linkK(w)$  
   and $\overline{\sigma}_{1},\overline{\sigma}_{2}\notin\linkK(w)$. 
\end{corollary} 

\begin{proof} 
Recall that a face $\sigma$ of a simplicial complex $K$ is a minimal missing 
face if and only if $\sigma\notin K$ but $\partial\sigma\subseteq K$. So by 
Lemma~\ref{mfstar}, 
$\partial\,\overline{\sigma}_{1},\partial\,\overline{\sigma}_{2} 
     \subseteq\starrK(w)$. 
By definition, neither $\overline{\sigma}_{1}$ nor $\overline{\sigma}_{2}$ 
have $w$ in their vertex sets, so neither do their boundaries. Therefore 
$\partial\,\overline{\sigma}_{1},\partial\,\overline{\sigma}_{2} 
     \subseteq\overline{K}\backslash w$. 
Therefore, as 
$\linkK(w)=\starrK(w)\cap\overline{K}\backslash w$, 
we have 
$\partial\,\overline{\sigma}_{1},\partial\,\overline{\sigma}_{2}\subseteq 
         \linkK(w)$. 

Also, as 
$\linkK(w)=\starrK(w)\cap\overline{K}\backslash w$, 
it cannot be that $\overline{\sigma}_{1},\overline{\sigma}_{2}$ are in 
$\linkK(w)$ as that would imply they are also 
in $\starrK(w)$, contradicting Lemma~\ref{mfstar}. 
\end{proof} 

One further observation we need regarding $\overline{K}$ is the following. 
Regarding $w$ as the $m^{th}$-vertex of $\overline{K}$, observe that 
$\overline{K}\backslash w$ is a simplicial complex on the vertex set $[m-1]$. 

\begin{lemma} 
   \label{Krsimplex} 
   There is an isomorphism of simplicial complexes 
   $\overline{K}\backslash w\cong\Delta^{m-2}$.  
\end{lemma} 

\begin{proof} 
 It is equivalent to show that $\overline{K}\backslash w$ has 
no missing faces. Suppose that $\sigma\in\Delta^{m-2}$ is a 
missing face of $\overline{K}\backslash w$. Then as 
$\overline{K}\backslash w$ is the restriction of $\overline{K}$ 
to the vertex set $[m-1]$, $\sigma$ is also a missing face of $\overline{K}$. 
On the other hand, as $MMF(\overline{K})=\{\sigma_{1},\sigma_{2}\}$, 
any missing face of $\overline{K}$ must have either $\sigma_{1}$ or $\sigma_{2}$ 
as a subface. Thus $\sigma$ must have either $\sigma_{1}$ or $\sigma_{2}$ 
as a subface. But this cannot happen since~$w$ is not in the vertex set 
of $\sigma$ but it is in the vertex sets of both $\sigma_{1}$ and $\sigma_{2}$. 
\end{proof}

\section{Moore's conjecture} 
\label{sec:moore} 

In this section we prove Theorems~\ref{intromacmoore} and~\ref{introellhyp} as 
consequences of Theorem~\ref{moore}. 

\begin{proposition} 
   \label{retract} 
   Let $K$ be a simplicial complex on the vertex set $[m]$ and let $X_{1},\ldots,X_{m}$ 
   be any sequence of pointed, path-connected $CW$-pairs. Suppose that 
   $\sigma_{1},\sigma_{2}\in MMF(K)$ and let $I$ and $J$ be the vertex 
   sets of $\sigma_{1}$ and $\sigma_{2}$ respectively. If $I\neq J$, $I\cup J=[m]$ 
   and $I\cap J\neq\emptyset$, then 
   $\cxx^{\partial\sigma_{1}}\vee\cxx^{\partial\sigma_{2}}$ 
   is a retract of $\cxx^{K}$. 
\end{proposition} 

\begin{proof} 
A new simplicial complex is introduced that will act as an intermediary. 
In general, a simplicial complex may be characterized by listing its minimal 
missing faces. Let $\overline{K}$ be the simplicial complex 
on the vertex set $[m]$ that is characterized by the condition that 
$MMF(\overline{K})=\{\sigma_{1},\sigma_{2}\}$. Intuitively, $\overline{K}$ is 
obtained from $K$ by filling in all missing faces that do not have either $\sigma_{1}$ 
or $\sigma_{2}$ as a subface. Rigorously, there is a map of simplicial complexes 
\(\namedright{K}{}{\overline{K}}\) 
that induces a map of polyhedral products 
\(\namedright{\cxx^{K}}{}{\cxx^{\overline{K}}}\).  
Since $\sigma_{1},\sigma_{2}$ are minimal missing faces of $K$, 
we have $\sigma_{1},\sigma_{2}\notin K$ but 
$\partial\sigma_{1},\partial\sigma_{2}\subseteq K$. The inclusion 
\(\namedright{\partial\sigma_{1}}{}{K}\) 
is a map of simplicial complexes and it induces a map of polyhedral products 
\(\namedright{\cxx^{\partial\sigma_{1}}}{}{\cxx^{K}}\). 
There is a similar map with respect to $\partial\sigma_{2}$. We will show that 
the composite 
\(\nameddright{\cxx^{\partial\sigma_{1}}\vee\cxx^{\partial\sigma_{2}}} 
      {}{\cxx^{K}}{}{\cxx^{\overline{K}}}\) 
has a left homotopy inverse. Note that this composite of polyhedral products 
is the same as the one induced by the inclusions 
\(\namedright{\partial\sigma_{1}}{}{\overline{K}}\) 
and 
\(\namedright{\partial\sigma_{2}}{}{\overline{K}}\), 
so it suffices to show that the map 
\(\namedright{\cxx^{\partial\sigma_{1}}\vee\cxx^{\partial\sigma_{2}}} 
      {}{\cxx^{\overline{K}}}\) 
has a left homotopy inverse.

The conditions on the vertex sets $I$ and $J$ imply that $\overline{K}$ has 
the same form as in Section~\ref{sec:combinatorics}. Relabelling the spaces 
$X_{1},\ldots,X_{m}$ if necessary, we may suppose that the intersection 
vertex $w$ corresponds to the $m^{th}$-coordinate space $X_{m}$. 
By Proposition~\ref{uxaunion}, the pushout of simplicial complexes 
in~(\ref{slr}) implies that there is a pushout of polyhedral products 
\begin{equation} 
  \label{slrpoly1} 
  \diagram
         \cxx^{\linkK(w)^{\circ}}\rto^-{g^{\circ}}\dto^{f^{\circ}} 
                 & \cxx^{\starrK(w)^{\circ}}\dto \\
         \cxx^{\overline{K}\backslash w^{\circ}}\rto & \cxx^{\overline{K}}.
  \enddiagram 
\end{equation} 
where $\linkK(w)^{\circ}$, $\starrK(w)^{\circ}$ 
and $\overline{K}\backslash w^{\circ}$ are $\linkK(w)$, 
$\starrK(w)$ and $\overline{K}\backslash w$ regarded 
as having vertex set $[m]$, and the maps $f^{\circ}$ and $g^{\circ}$ are 
induced by the inclusions 
\(\namedright{\linkK(w)^{\circ}}{}{\overline{K}\backslash w^{\circ}}\) 
and 
\(\namedright{\linkK(w)^{\circ}}{}{\starrK(w)^{\circ}}\) 
respectively. The vertex sets of $\linkK(w)$ and 
$\overline{K}\backslash w$ are both $[m-1]$, so by the definition of the 
polyhedral product, 
\[\cxx^{\linkK(w)^{\circ}}=\cxx^{\linkK(w)}\times X_{m} 
      \qquad\cxx^{\overline{K}\backslash w^{\circ}}= 
      \cxx^{\overline{K}\backslash w}\times X_{m}\] 
and $f^{\circ}=f\times 1$ where $f$ is induced by the inclusion 
\(\namedright{\linkK(w)}{}{\overline{K}\backslash w}\) 
and $1$ is the identity map on $X_{m}$. On the other hand, the vertex set of 
$\starrK(w)$ is $[m]$ so 
$\starrK(w)^{\circ}=\starrK(w)$. 
Since $\starrK(w)=\linkK(w)\ast\{w\}$, 
the definition of the polyhedral product implies that
\[\cxx^{\starrK(w)^{\circ}}=\cxx^{\linkK(w)}\times CX_{m}\] 
and $g^{\circ}=1\times i_{m}$ where~$1$ is the identity map on 
$\cxx^{\linkK(w)}$ and 
\(i_{m}\colon\namedright{X_{m}}{}{CX_{m}}\) 
is the inclusion of the base of the cone. Putting all this together, the 
pushout~(\ref{slrpoly1}) becomes the pushout 
\begin{equation} 
  \label{slrpoly2} 
  \diagram 
       \cxx^{\linkK(w)}\times X_{m}\rto^-{1\times i_{m}}\dto^{f\circ 1} 
           & \cxx^{\linkK(w)}\times CX_{m}\dto \\ 
       \cxx^{\overline{K}\backslash w}\times X_{m}\rto & \cxx^{\overline{K}}. 
  \enddiagram 
\end{equation}  
By Lemma~\ref{Krsimplex}, $\overline{K}\backslash w\cong\Delta^{m-2}$, 
so by Lemma~\ref{2egs}~(a), $\cxx^{\overline{K}\backslash w}=\prod_{i=1}^{m-1} CX_{i}$. 
Therefore, in~(\ref{slrpoly2}), both $\cxx^{\overline{K}\backslash w}$ and $CX_{m}$ 
are contractible, implying that~(\ref{slrpoly2}) is equivalent, up to homotopy, 
to the homotopy pushout 
\begin{equation} 
  \label{slrpoly3} 
  \diagram 
       \cxx^{\linkK(w)}\times X_{m}\rto^-{\pi_{1}}\dto^{\pi_{2}} 
           & \cxx^{\linkK(w)}\dto \\ 
       X_{m}\rto & \cxx^{\overline{K}}  
  \enddiagram 
\end{equation}          
where $\pi_{1}$ and $\pi_{2}$ are the projections onto the first and second 
factors respectively. It is well known that the pushout of the projections 
\(\namedright{A\times B}{}{A}\) 
and 
\(\namedright{A\times B}{}{B}\) 
is homotopy equivalent to the join of $A$ and $B$, which in turn is 
homotopy equivalent to $\Sigma A\wedge B$. So~(\ref{slrpoly3}) 
implies that there is a homotopy equivalence   
\begin{equation} 
  \label{linkjoin} 
  \cxx^{\overline{K}}\simeq\Sigma\cxx^{\linkK(w)}\wedge X_{m}. 
\end{equation}  

Now consider the minimal missing faces $\sigma_{1}$ and $\sigma_{2}$ 
of $\overline{K}$. As in Section~\ref{sec:combinatorics}, let 
$\overline{\sigma}_{1},\overline{\sigma}_{2}$ be the restrictions of 
$\sigma_{1},\sigma_{2}$ respectively to the vertex sets 
$\overline{I}=I\backslash\{w\}, \overline{J}=J\backslash\{w\}$. Note that as 
$I\neq J$ we also have $\overline{I}\neq\overline{J}$. By Corollary~\ref{mflink},  
$\overline{\sigma}_{1},\overline{\sigma}_{2}\notin\linkK(w)$ but 
$\partial\,\overline{\sigma}_{1},\partial\,\overline{\sigma}_{2}\subseteq 
      \linkK(w)$. 
Therefore, the full subcomplex of $\linkK(w)$ on $\overline{I}$ 
is $\partial\,\overline{\sigma}_{1}$, and the full subcomplex of 
$\linkK(w)$ on $\overline{J}$ is 
$\partial\,\overline{\sigma}_{2}$. By Proposition~\ref{fullsubcomplex}, 
this implies that $\cxx^{\partial\overline{\sigma}_{1}}$ and 
$\cxx^{\partial\overline{\sigma}_{2}}$ are retracts of 
$\cxx^{\linkK(w)}$. By~\cite[Theorem 2.21]{BBCG1}, the fact that 
$\partial\,\overline{\sigma}_{1}$ and $\partial\,\overline{\sigma}_{2}$ 
are full subcomplexes of $\linkK(w)$ on different index 
sets implies that that 
$\Sigma\cxx^{\partial\,\overline{\sigma}_{1}}\vee\Sigma\cxx^{\partial\,\overline{\sigma}_{2}}$ 
is a retract of $\Sigma\cxx^{\linkK(w)}$. Thus~(\ref{linkjoin}) 
implies that 
$(\Sigma\cxx^{\partial\,\overline{\sigma}_{1}}\wedge X_{m})\vee 
      (\Sigma\cxx^{\partial\,\overline{\sigma}_{2}}\wedge X_{m})$ 
is a retract of $\cxx^{\overline{K}}$. 

We wish to choose the retraction more carefully. Restrict $\overline{K}$ to the 
full subcomplex on the vertex set $I$. Then 
$MMF(\overline{K}_{I})=\{\sigma_{1}\}$, so $\overline{K}_{I}=\partial\sigma_{1}$. 
Therefore the star-link-restriction pushout for $\overline{K}_{I}$ with respect 
to the vertex $w$ becomes  
\[\diagram 
        \partial\,\overline{\sigma}_{1}\rto\dto 
              & \partial\,\overline{\sigma}_{1}\ast\{w\}\dto \\ 
        \partial\sigma_{1}\backslash w\rto & \partial\sigma_{1}.  
  \enddiagram\] 
Note that $\partial\sigma_{1}\backslash w$ is the simplex $\Delta^{k-1}$ 
on the vertex set $\{i_{1},\ldots,i_{k}\}$. Now arguing as for~(\ref{slrpoly1}) -- 
(\ref{slrpoly3}) and equation~(\ref{linkjoin}), we obtain in place of~(\ref{linkjoin}) 
a homotopy equivalence 
$\cxx^{\overline{K}_{I}}=\cxx^{\partial\sigma_{1}}\simeq 
      \Sigma\cxx^{\partial\,\overline{\sigma}_{1}}\wedge X_{m}$. 
Thus we may choose the map 
\(\namedright{\Sigma\cxx^{\partial\,\overline{\sigma}_{1}}\wedge X_{m}} 
      {}{\cxx^{\overline{K}}}\) 
as the composite 
\(\nameddright{\Sigma\cxx^{\partial\,\overline{\sigma}_{1}}\wedge X_{m}}{\simeq} 
      {\cxx^{\partial\sigma_{1}}}{}{\cxx^{\overline{K}}}\). 
Doing the same for $\partial\sigma_{2}$ we obtain a composite 
\(\nameddright{(\Sigma\cxx^{\partial\,\overline{\sigma}_{1}}\wedge X_{m})\vee 
      (\Sigma\cxx^{\partial\,\overline{\sigma}_{2}}\wedge X_{m})}{\simeq}     
      {\cxx^{\partial\sigma_{1}}\vee\cxx^{\partial\sigma_{2}}}{}{\cxx^{\overline{K}}}\), 
and it is this composite that has a left homotopy inverse. In particular, we have 
produced a left homotopy inverse for the map 
\(\namedright{\cxx^{\partial\sigma_{1}}\vee\cxx^{\partial\sigma_{2}}}{} 
      {\cxx^{\overline{K}}}\) 
induced by the inclusions 
\(\namedright{\partial\sigma_{1}}{}{\overline{K}}\) 
and 
\(\namedright{\partial\sigma_{2}}{}{\overline{K}}\), 
as required.  
\end{proof} 

Recall that for $1\leq i\leq m$, $Y_{i}$ is the homotopy fibre of the inclusion 
\(\namedright{A_{i}}{}{X_{i}}\). 

\begin{theorem} 
   \label{moore} 
   Let $K$ be a simplicial complex on the vertex set $[m]$ and let $\uxa$ 
   be any sequence of pointed, path-connected $CW$-pairs. The following hold: 
   \begin{letterlist} 
      \item if $MMF(K)=\{\sigma_{1},\ldots,\sigma_{n}\}$ and these minimal 
                missing faces are mutually disjoint, then there is a homotopy 
                equivalence 
                \[\Omega\uxa^{K}\simeq(\prod_{i=1}^{m}\Omega X_{i})\times  
                         (\prod_{j=1}^{n}\Omega\cyy^{\partial\sigma_{j}});\] 
      \item if $\sigma_{1}$ and $\sigma_{2}$ are minimal missing faces of $K$ 
                with nontrivial intersection then 
                $\Omega\left(\cyy^{\partial\sigma_{1}}\vee\cyy^{\partial\sigma_{2}}\right)$ 
                retracts off $\Omega\uxa^{K}$. 
   \end{letterlist} 
\end{theorem} 
 
\begin{proof} 
By Theorem~\ref{cyysplit}, there is a homotopy equivalence 
\begin{equation} 
  \label{uxadecomp} 
  \Omega\uxa^{K}\simeq(\prod_{i=1}^{m}\Omega X_{i})\times\Omega\cyy^{K} 
\end{equation}  
where, for $1\leq i\leq m$, $Y_{i}$ is the homotopy fibre of the inclusion 
\(\namedright{A_{i}}{}{X_{i}}\). 

If all of the minimal missing faces of $K$ are mutually disjoint then 
there is a simplicial isomorphism 
$K\cong K_{0}\ast K_{1}\ast\cdots\ast K_{n}$ where $K_{0}$ is a product of simplices  
and, for $1\leq j\leq n$, $K_{j}=\partial\sigma_{j}$ (a proof of this may 
be found in~\cite{BBCG2}, although it may be more commonly known). 
In general, the definition of a polyhedral product implies that there is a homeomorphism 
$\uxa^{L\ast M}\cong\uxa^{L}\times\uxa^{M}$. In our case, as $K_{0}$ 
is a simplex, Lemma~\ref{2egs}~(a) implies that $\cyy^{K_{0}}$ is a 
product of cones and so is contractible. Thus 
\[\cyy^{K}\simeq\cyy^{K_{1}}\times\cdots\times\cyy^{K_{n}}= 
       \cyy^{\partial\sigma_{1}}\times\cdots\times\cyy^{\partial\sigma_{n}}.\] 
Combining this with~(\ref{uxadecomp}), the homotopy decomposition in 
part~(a) follows. 

Next, suppose that $\sigma_{1}$ and $\sigma_{2}$ are minimal missing faces 
of $K$ that intersect nontrivially. Let $I$ and~$J$ be the vertex sets of $\sigma_{1}$ 
and $\sigma_{2}$ respectively. Let $K_{I\cup J}$ be the full subcomplex of $K$ 
on the index set $I\cup J$. By Proposition~\ref{fullsubcomplex}, 
$\cyy^{K_{I\cup J}}$ is a retract of $\cyy^{K}$. Further, Proposition~\ref{retract} 
implies that $\cyy^{\partial\sigma_{1}}\vee\cyy^{\partial\sigma_{2}}$ is a 
retract of $\cyy^{K_{I\cup J}}$. Hence 
$\cyy^{\partial\sigma_{1}}\vee\cyy^{\partial\sigma_{2}}$ is a 
retract of $\cyy^{K}$. Combining this with~(\ref{uxadecomp}), the assertion 
in part~(b) follows. 
\end{proof} 

We now turn to Moore's Conjecture and the distinguishing of elliptic and 
hyperbolic spaces. For Theorem~\ref{intromacmoore}, we assume that 
each pair $(X_{i},A_{i})$ is $(D^{n_{i}},S^{n_{i}-1})$ for $n_{i}\geq 2$. Note that 
the homotopy fibre $Y_{i}$ of the inclusion 
\(\namedright{S^{n_{i}-1}}{}{D^{n_{i}}}\) 
is also $S^{n_{i}-1}$, so the pair $(CY_{i},Y_{i})$ in Theorem~\ref{moore} is also homotopy 
equivalent to $(D^{n_{i}},S^{n_{i}-1})$. Note that as each~$X_{i}$ is $D^{n_{i}}$, the term 
$\prod_{i=1}^{m}\Omega X_{i}$ in Theorem~\ref{moore}~(a) is contractible. 
Also, by Lemma~\ref{2egs}~(b), each term $\cyy^{\partial\sigma_{i}}$ in 
Theorem~\ref{moore}~(a) and~(b) is homotopy equivalent to a simply-connected sphere. 

\begin{proof}[Proof of Theorem~\ref{intromacmoore}]  
Theorem~\ref{moore}~(b) implies that if $K$ has two minimal missing 
faces with nontrivial intersection then a wedge of two simply-connected spheres 
retracts off $\Omega\uxa^{K}$. The Hilton-Milnor Theorem shows that a 
wedge of two such spheres is hyperbolic, and Neisendorfer and Selick~\cite{NS} 
showed that a wedge of two such spheres has no exponent at any prime $p$. Hence 
Moore's conjecture holds in this case. On the other hand, if all the minimal 
missing faces of $K$ are mutually disjoint then Theorem~\ref{moore}~(a) 
implies that $\Omega\uxa^{K}$ is homotopy equivalent to a finite 
product of spheres. This is elliptic, and as each sphere has an exponent at 
every prime $p$, so does a finite product of them. Hence Moore's Conjecture 
holds in this case as well. 
\end{proof} 

\begin{proof}[Proof of Theorem~\ref{introellhyp}] 
Recall that if $Y$ is any space then $\Sigma Y$ is rationally homotopy 
equivalent to a wedge of spheres. In particular, if $\partial\sigma\subseteq K$ 
and each $Y_{i}$ is rationally nontrivial 
then by Lemma~\ref{2egs}~(b) the space $\cyy^{\partial\sigma}$ is rationally 
homotopy equivalent to a wedge of simply-connected spheres. Thus if $v$ is a vertex 
of $\partial\sigma$ and $\mbox{rank}(\pi_{\ast}(Y_{v})\otimes\mathbb{Q})\geq 2$ 
then $\cyy^{\partial\sigma}$ is rationally homotopy equivalent to a wedge of at 
least two simply-connected spheres. 

Suppose that $\uxa^{K}$ is elliptic. The homotopy decomposition  
$\Omega\uxa^{K}\simeq(\prod_{i=1}^{m}\Omega X_{i})\times\Omega\cyy^{K}$ 
in Theorem~\ref{cyysplit} then immediately implies that each $X_{i}$ must be 
elliptic, so condition~(i) holds. This homotopy decomposition also implies 
that $\cyy^{K}$ is elliptic. Let $\sigma_{1},\ldots,\sigma_{n}$ be the 
minimal missing faces of $K$. If two of these minimal missing faces intersect, 
say $\sigma_{1}$ and $\sigma_{2}$, then Theorem~\ref{moore} implies that 
$\Omega\left(\cyy^{\partial\sigma_{1}}\vee\cyy^{\partial\sigma_{2}}\right)$ 
retracts off $\Omega\cyy^{K}$. Since each of $\cyy^{\partial\sigma_{1}}$ 
and $\cyy^{\partial\sigma_{2}}$ is rationally homotopy equivalent to a wedge 
of simply-connected spheres, the space 
$\cyy^{\partial\sigma_{1}}\vee\cyy^{\partial\sigma_{2}}$ 
is rationally homotopy equivalent to a wedge of at least two simply-connected 
spheres, implying that it is hyperbolic. Therefore $\cyy^{K}$ is hyperbolic, a 
contradiction. Hence the minimal missing faces of $K$ must be mutually disjoint, 
implying that condition~(ii) holds. Because condition~(ii) holds, Theorem~\ref{moore} 
implies that 
$\Omega\cyy^{K}\simeq\prod_{j=1}^{n}\Omega\cyy^{\partial\sigma_{j}}$. 
It has already been observed that if $v$ is a vertex of $\partial\sigma_{j}$ 
and $\mbox{rank}(\pi_{\ast}(X_{v})\otimes\mathbb{Q})\geq 2$ then 
$\cyy^{\partial\sigma_{j}}$ is rationally homotopy equivalent to a wedge of at 
least two simply-connected spheres, and so is hyperbolic, implying that $\cyy^{K}$ 
is hyperbolic, a contradiction. Thus condition~(iii) holds. 

Conversely, suppose that conditions (i) to (iii) hold. By Theorem~\ref{moore}, 
condition~(ii) implies that 
$\Omega\uxa^{K}\simeq 
    (\prod_{i=1}^{m}\Omega X_{i})\times(\prod_{j=1}^{n}\Omega\cyy^{\partial\sigma_{j}})$, 
where $\sigma_{1},\ldots,\sigma_{n}$ are the minimal missing faces of~$K$.  
For each vertex $v$ of any $\sigma_{i}$, condition~(iii) states that $Y_{v}$ is rationally 
homotopy equivalent to a sphere. Therefore Lemma~\ref{2egs}~(b) implies that  
$\cyy^{\partial\sigma_{i}}$ is rationally homotopy equivalent to a sphere.  
As each $X_{i}$ is elliptic by condition~(i), it has finitely many rational homotopy 
groups. Hence the homotopy decomposition for $\Omega\uxa^{K}$ implies 
that $\uxa^{K}$ has finitely many rational homotopy groups and so is elliptic. 
\end{proof}

\bibliographystyle{amsplain}

\begin{thebibliography}{22} 
\bibitem{A} D.J. Anick, Homotopy exponents for spaces of category two, 
   \emph{Algebraic Topology (Arcata, CA, 1986)}, pp. 24-52, Lecture Notes in Math. \textbf{1370}, 
   Springer, Berlin, 1989.  
\bibitem{BBCG1} A. Bahri, M. Bendersky, F.R. Cohen, and S. Gilter,
   The polyhedral product functor: a method of decomposition for moment-angle
   complexes, arrangements and related spaces, \emph{Adv. Math.} \textbf{225}
   (2010), 1634-1668. 
\bibitem{BBCG2} A. Bahri, M. Bendersky, F.R. Cohen, and S. Gilter, Free loop 
   spaces of toric spaces, \emph{Bol. Soc. Mat. Mexicana} \textbf{23} (2017), 257-265. 
\bibitem{BP} V.M. Buchstaber and T.E. Panov, \emph{Toric topology}, 
   Mathematical Surveys and Monographs \textbf{204}, American Mathematical 
   Society, 2015. 
\bibitem{CPSS} W. Chach\'{o}lski, W. Pitsch, J. Scherer and D. Stanley, 
   Homotopy exponents for large $H$-spaces, \emph{Int. Math. Res. Not. IMRN 2008}, 
   \textbf{16}, Art. ID rnn061, 5pp. 
\bibitem{DJ} M.W. Davis and T. Januszkiewicz, Convex polytopes, Coxeter orbifolds 
   and torus actions, \emph{Duke Math. J.} \textbf{62} (1991), 417-452.
\bibitem{DS} G. Denham and A. Suciu, Moment-angle complexes,
   monomial ideals and Massey products, \emph{Pure Appl. Math Q.}
   \textbf{3} (2007), 25-60. 
\bibitem{FHT} Y. F\'{e}lix, S. Halperin and J.-C. Thomas, Elliptic spaces II. 
   \emph{Enseign. Math. (2)} \textbf{39} (1993), 25-32. 
\bibitem{FT} Y. F\'{e}lix and D. Tanr\'{e}, Rational homotopy of the
   polyhedral product functor, \emph{Proc. Amer. Math. Soc.}
   \textbf{137} (2009), 891-898. 
\bibitem{GT1} J. Grbi\' c and S. Theriault, The homotopy type of the
   complement of a coordinate subspace arrangement,
   \emph{Topology} \textbf{46} (2007), 357-396. 
\bibitem{GT2} J. Grbi\'{c} and S. Theriault, The homotopy type of 
   the polyhedral product for shifted complexes, \emph{Adv. Math.} \textbf{245} 
   (2013), 690-715. 
\bibitem{IK} K. Iriye and D. Kishimoto, Decompositions of polyhedral 
   products, \emph{Adv. Math.} \textbf{245} (2013), 716-736. 
\bibitem{J} I.M. James, The suspension triad of a sphere, \emph{Ann. of Math.} 
   \textbf{63} (1956), 407-429. 
\bibitem{Lo} J. Long, Thesis, Princeton University, 1978. 
\bibitem{Mac} D. McGavran, Adjacent connected sums and torus actions, 
   \emph{Trans. Amer. Math. Soc.} \textbf{251} (1979), 235-254. 
\bibitem{MW} C.A. McGibbon and C.W. Wilkerson, Loop spaces of finite 
   complexes at large primes, \emph{Proc. Amer. Math. Soc.} \textbf{96} (1986), 698-702. 
\bibitem{N} J.A. Neisendorfer, The exponent of a Moore space. \emph{Algebraic 
   topology and algebraic $K$-theory (Princeton, N.J., 1983)}, 35-71, 
   Ann. of Math. Stud. \textbf{113}, Princeton Univ. Press, Princeton, NJ, 1987. 
\bibitem{NS} J. Neisendorfer and P. Selick, Some examples of spaces with 
   and without homotopy exponents, \emph{Current trends in algebraic topology, 
   Part 1 (London, Ont., 1981)}, pp. 343-357, CMS Conf. Proc. \textbf{2}, 
   Amer. Math. Soc., Providence, R.I., 1982. 
\bibitem{P} G.J. Porter, The homotopy groups of wedges of suspensions,
   \emph{Amer. J. Math.} \textbf{88} (1966), 655-663. 
\bibitem{Se} P. Selick, On conjectures of Moore and Serre in the case of 
   torsion-free suspensions, \emph{Math. Proc. Cambridge Philos. Soc.} 
   \textbf{94} (1983), 53-60. 
\bibitem{St} M. Stelzer, Hyperbolic spaces at large primes and a conjecture 
   of Moore, \emph{Topology} \textbf{43} (2004), 667-675.  
\bibitem{To} H. Toda, On the double suspension $E^{2}$, \emph{J. Inst. 
   Polytech. Osaka City Univ., Ser. A} \textbf{7} (1956), 103-145. 
\end{thebibliography}

\end{document}